\documentclass[12pt]{amsart}
\usepackage{amsmath,amsfonts,amssymb}
\usepackage{color,graphicx}
\usepackage[left=2.5cm,right=2.5cm,top=3.5cm,bottom=3cm]{geometry}

\usepackage{amsthm}
\usepackage{ dsfont }
\newtheorem{thm}{Theorem}[section]
\newtheorem{lem}[thm]{Lemma}
\newtheorem{defn}[thm]{Definition}
\newtheorem{cor}[thm]{Corollary}
\newtheorem{prop}[thm]{Proposition}

\newtheorem{ex}[thm]{Example}

\theoremstyle{definition}
\newtheorem{rem}[thm]{Remark}

\usepackage[utf8]{inputenc}

\newcommand{\N}{\mathds{N}}

\newcommand{\R}{\mathds{R}}

\date{\today}
\title{ Non-linear operators and differentiability of Lipschitz functions}
\author{Mohammed Bachir, Sebastián Tapia-García}

\address{Laboratoire SAMM 4543, Université Paris 1 Panthéon-Sorbonne
Centre P.M.F. 90 rue Tolbiac, 75634 Paris cedex 13, France.}
\email{Mohammed.Bachir@univ-paris1.fr}
\address{ Departamento de Ingenier\'ia Matem\'atica, CMM (CNRS UMI 2807) Universidad de Chile\\ Beauchef 851, Santiago, Chile.}
\address{Institute de Math\'ematique de Bordeaux, IMB (CNRS UMR 5251) Universit\'e de Bordeaux, Course de la liberation 351, Talence, France}
\email{stapia@dim.uchile.cl}
\begin{document}

\maketitle

\begin{abstract} 
In this work we provide a characterization of distinct type of  (linear and non-linear) maps between Banach spaces in terms of the differentiability of certain class of Lipschitz functions.
Our results are stated in an abstract bornological and non-linear framework. 
Restricted to the linear case, we can apply our results to compact, weakly-compact, limited and completely continuous linear operators.
Moreover, our results yield a characterization of Gelfand-Phillips spaces and recover some known result of Schur spaces and reflexive spaces concerning the differentiability of real-valued Lipschitz functions.
\end{abstract}
\noindent \textbf{Key words:} Linear and non-linear operators, differentiability of Lipschitz functions, bornology, weakly compact operators, completely continuous operators.\\
\noindent \textbf{MSC 2020:} Primary: 
46A17, 26A16, 47B07. Secondary: 47B38, 49J50.
\tableofcontents

\section{Introduction}

Differentiability of convex and Lipschitz functions defined on general (infinite dimensional) Banach spaces has attracted the interest of researchers since the last century.
Distinct notions of differentiability and their relation with the geometry of the ambient space is one of the most common topics in the area. 
A nice survey about differentiability of convex functions can be found in \cite{BV}. 
A general approach of differentiability and its applications to renorming can be found in \cite{DGZ}. 
Particularly, the relation about G\^ateaux, weak-Hadamard and Fr\'echet differentiability has been a prolific subject for many authors. 
For instance, see \cite{BL,BF,BV} and reference therein. \\

In order to motivate this work, let us proceed with the following definition.

\begin{defn}
	Let $X$ be a Banach space. 
	A bounded subset $A$ of $X$ is called limited if for any weak${}^*$ null sequence $(x_n)^*_n$, the following limit holds:
	\[  \lim_{n\to\infty}\sup_{x\in A} |\langle x^*_{n},x \rangle| = 0. \]
	That is, sequentially weakly${}^*$ convergence is uniform on $A$.
\end{defn}

Let $X$ and $Y$ be two Banach spaces. A linear bounded operator $T\in\mathcal{L}(Y,X)$ is called limited if $T(B_Y)$ is limited. 
Further information about limited sets and limited operators can be found in \cite{BFT,BD} and references therein.
In \cite{B,BFT}, we can find how to characterize limited and compact operators in terms of the differentiability of a certain class of functions respectively. 
More precise, the main result of \cite{B} read as follows.

\begin{thm}\cite[Theorem 1]{B}\label{limited differentiability}
	Let $X$ and $Y$ be two real Banach spaces, let $\mathcal{U}$ be a nonempty open convex subset of $X$ and let $T\in\mathcal{L}(Y,X)$. 
	Then, $T$ is limited if and only if for every continuous convex function $f:\mathcal{U}\to \R$, $f\circ T$ is Fr\'echet differentiable at $y\in Y$ whenever $f$ is G\^ateaux differentiable at $Ty\in \mathcal{U}$. 
\end{thm}
 
In \cite{BFT} it can be found a similar result, but interchanging limited by compact operators and convex by Lipschitz functions. 
The mentioned result of \cite{BFT} exploits the following well known result: G\^ateaux and Hadamard differentiability coincide for Lipschitz functions defined on open subsets of a Banach space. 
In fact, as we show in Section~\ref{limited section}, we point out that Theorem~\ref{limited differentiability} exploits the following fact: G\^ateaux and limited differentiability coincide for convex functions (see Proposition~\ref{convex differentiability}). 
For definitions see Section~\ref{notation}.\\

In this paper we obtain results with the same structure as Theorem~\ref{limited differentiability}. 
More precisely, we look for which classes of (linear and non-linear) operators can be characterized in terms of the differentiability of a given set of Lipschitz functions.
In fact, as a consequence of our work we provide an answer to \cite[Question 5.7]{BFT}, where the authors asked about a possible characterization for linear weakly-compact operators in terms of the differentiability of some class of functions.
In order to state the main applications of this paper, let us proceed with the following definitions:

\vskip5mm

Let $X$, $Y$ be two real Banach spaces and let $T\in\mathcal{L}(Y,X)$.  
We say that $T$ is weakly compact if, $T$ sends the closed unit ball onto a relatively weakly compact set. 
We say that $T$ is completely continuous or Dunford-Pettis operator if, for every weakly convergent sequence $(y_{n})$ from $Y$, the sequence $(Ty_{n})$ is norm-convergent in $X$ (see \cite[~Definition VI.3.2,  ~p. 173]{C}). 
Equivalently, $T$ is completely continuous if, for every relatively weakly compact subset $K$ of $Y$, we have that $T(K)$ is relatively compact. 
It is well known that every compact operator is completely continuous (see \cite[p. 143]{Banach}).\\

The main result of this paper, Theorem~\ref{beta differentiability}, yields the following characterizations of weakly compact operators and completely continuous operators, beside of recovering the main result of \cite{BFT} for compact operators.

\begin{thm}\label{weak hadamard differentiability}
	Let $X$ and $Y$ be real Banach spaces and let $T\in \mathcal{L}(Y, X)$. 
	Then, $T$ is weakly compact if and only if for every Lipschitz function $f:X\to\R$, $f\circ T$ is Fr\'echet differentiable at $y$ whenever $f$ is weak-Hadamard differentiable at $Ty$.
	
\end{thm}

\begin{thm}\label{Gateaux-weak hadamard differentiability}
	Let $X$ and $Y$ be real Banach spaces and let $T\in \mathcal{L}(Y, X)$.  
	Then, $T$ is completely continuous if and only if for every Lipschitz function $f:X\to\R$, $f\circ T$ is weak-Hadamard  differentiable at $y$ whenever $f$ is G\^ateaux  (equivalent to Hadamard) differentiable at $Ty$.	
\end{thm}

  Our main result, Theorem~\ref{beta differentiability}, is an unified statement that characterizes several types of linear and non-linear operators in terms of the {\it "p.h. differentiability"} of certain Lipschitz functions (see Definition~\ref{ph differentiability}), among which weakly-compact, completely continuous, compact and limited linear operators are included. 
 In order to state the mentioned theorem, we introduce an abstract property on bornologies that we call property $(S)$ in Definition~\ref{property S}, which is satisfied by the Hadamard, weak-Hadamard and limited bornology. 
 Also, we propose a weaker version of differentiability which can be applied to positively homogeneous maps, see Definition~\ref{ph differentiability}.\\

The outline of this paper is as follows: In Section~\ref{notation} we give definitions and fix the notation used through this manuscript.  
In Section~\ref{preliminar results} we present some examples of bornologies that satisfies property $(S)$, see Definition~\ref{property S}. 
Also, for a vector bornology $\beta$ on $X$ which satisfies property $(S)$, different from the Fr\'echet bornology, we construct a Lipschitz $f:X\to \R$ which is $\beta$-differentiable at $0$ but not Fr\'echet differentiable at $0$. 
In Section~\ref{main result}, we state and prove our main result Theorem~\ref{beta differentiability}. 
Moreover, we also prove Theorem~\ref{weak hadamard differentiability} and Theorem~\ref{Gateaux-weak hadamard differentiability}. 
In Section~\ref{limited section}, we present an alternative proof of Theorem~\ref{limited differentiability} based in the fact that G\^ateaux and limited differentiability coincide for continuous convex functions (see Proposition~\ref{convex differentiability}).
Finally, in Section~\ref{consequences}, we present some consequences of our results. 
More precisely, as a direct consequence of Theorem~\ref{beta differentiability} we characterize Gelfand-Phillips spaces, finite dimensional spaces and we recover some well-known characterizations of  Schur spaces and reflexive spaces.
Also, we present some results on spaceability and an extension of a result of Bourgain and Diestel given in \cite{BD}.

\section{Preliminaries and notations}\label{notation}

The notation used through this paper is standard. By $X$, $Y$ and $Z$ we denote real Banach spaces.
By $X^*$ we denote the dual space of $X$.
 For $x\in X$ and $r>0$ we denote by $B(x,r)$ and $B_X$ the open ball centered at $x$ of radius $r$ and the open unit ball respectively.
We denote by $\mathcal{L}(Y,X)$ the vector space of linear bounded operators from $Y$ to $X$.
A set $A\subset X$ is said balanced if $\lambda A\subset A$, for all $|\lambda|\leq1$.
We recall that a bornology $\beta$ on $X$ is a family of bounded nonempty subsets of $X$ satisfying the following three properties:
\begin{enumerate}
	\item $\bigcup_{A\in\beta} A=X $,
	\item for all $A,~B\in\beta$, $A\cup B\in \beta$, and
	\item for all $A\in \beta$ and $B\subset A$,~ $B\in \beta$.
\end{enumerate}

If $A\in\beta$, we say that $A$ is a $\beta$-set.
For instance, some known bornologies are the ones of G\^ateaux, Hadamard, Limited, weak-Hadamard and Fr\'echet. 
These bornologies correspond to the family of finite sets, relatively compact set, limited set, relatively weakly-compact set and bounded sets respectively. 
For a bornology $\beta$ on $X$ and a linear operator $T:Y\to X$, we say that $T$ is a $\beta$-operator if $T(B(y,r))\in \beta$, for all $y\in Y$ and all $r>0$.
Since $\beta$ only contains bounded sets, any $\beta$-operator is continuous. 
We recall that a  function $f: X\to Z$ is $L$-Lipschitz, for $L\geq 0$, if  
\[\|f(x_1)-f(x_2)\|\leq L\|x_1-x_2\|~ \text{for every }x_1,~x_2\in X.\] 
The Lipschitz constant of $f$, denoted by $\textup{Lip}(f)$, is the smallest value $L\geq 0$ such that $f$ is $L$-Lipschitz.
The space of real-valued Lipschitz functions defined on $X$ is denoted by $\textup{Lip}(X)$.
The Lipschitz constant is a seminorm on $\textup{Lip}(X)$.
Finally, for a bornology $\beta$ on $X$, we recall that a function $f:X\to Z$ is said $\beta$-differentiable at $x_0\in X$, with differential $d_\beta f(x_0)\in \mathcal{L}(X,Z)$, if 
\[\lim_{t\to 0}\sup_{u\in A}\left\| \dfrac{f(x_0+tu)-f(x_0)}{t}- d_\beta f(x_0)(u) \right\| =0, ~ \text{for all }A\in \beta.\]

Let us introduce the following definitions which are needed to state our main theorem.

\begin{defn} \label{ball}
Let $\beta$ be a bornology on $X$. 
We say that $\beta$ is a vector bornology if:
\begin{enumerate}
	\item For any $A\in \beta$, $\textnormal{bal}(A):=\{\lambda x:~x \in A,~|\lambda|\leq1\}\in \beta$.
	\item For any $A\in \beta$, $\lambda\in \R$ and $x\in X$, $x+\lambda A\in \beta$,
\end{enumerate}
where $\textnormal{bal}(A)$ stands for the balanced hull of $A$.
\end{defn}
Notice that the Hadamard, weakly-Hadamard, limited and Fr\'echet bornologies are clearly vector bornologies. Let $\beta$ be a vector bornology on $X$ and let $T\in \mathcal{L}(Y,X)$. 
Then, $T$ is a $\beta$-operator if and only if $T(B_Y)\in \beta$.
Further, a function $f:X\to Z$ is $\beta$-differentiable at $x_0\in X$, with differential $d_\beta f(x_0)\in \mathcal{L}(X,Z)$, if 
\[\lim_{t\to 0^+}\sup_{u\in A}\left\| \dfrac{f(x_0+tu)-f(x_0)}{t}- d_\beta f(x_0)(u) \right\| =0, ~ \text{for all }A\in \beta.\]

We will sometimes simply denote  $\beta=G, H, L, wH$ and $F$, for the G\^ateaux, Hadamard, limited, weak Hadamard and Fr\'echet bornologies respectively, where the underlying Banach space is clear from the context. 
With these notations, $\beta$-differentiability corresponds  to the G\^ateaux (resp. Hadamard, limited, weak Hadamard and Fr\'echet) differentiability.
Let us continue with the definition of property $(S)$ which plays a fundamental roll in this manuscript.

\begin{defn}\label{property S}
Let $\beta$ be a vector bornology on $X$. 
We say that $\beta$ satisfies property $(S)$ if for every bounded set $A\subset X$ such that $A\notin \beta$, there are a sequence $(x_n)_n\subset A$ and $\delta >0$ such that for any increasing sequence $(n_k)\subset \N$ and for any sequence $(y_k)_k\subset X$ satisfying $\|y_k-x_{n_k}\| \leq \delta$ for all $k\in\N$, the set $\{y_k:k\in\N\}\notin \beta$.
\end{defn}
We say that a set $A\subset X$ is semi-balanced if $\lambda A\subset A$, for all $0\leq \lambda\leq 1$.
\begin{prop} \label{sequence property (S)}
 Let $\beta$ be a vector bornology on $X$ satisfying property~$(S)$. 
 Let $A\subset X$ be a semi-balanced set such that $A\notin \beta$. 
 Let $(x_n)_n\subset A$ and $\delta>0$ given by property $(S)$.
 Then, $(x_n)_n$ does not have accumulation points. 
 Moreover, the sequence $(x_n)_n$ can be chosen such that $\|x_n\|=\|x_1\|$ for all $n\in \N$.
\end{prop}

\begin{proof}
	Reasoning by contradiction, let us assume that there is a subsequence $(x_{n_k})_k$ of $(x_n)_n$ convergent to some $\overline{x}\in X$. 
	Up to a subsequence, we can assume that $\| \overline{x}-x_{n_k}\| <\delta/2$ for all $ k\in \N$. 
	Then, the constant sequence $(y_k)_k\subset  X$, defined by $y_k=\overline{x}$ for all $k\in\N$,
	satisfies $\|y_k-x_{n_k}\| \leq \delta$ for all $k\in\N$. 
	Thus, by property $(S)$, $\{y_k:~k\in\N\}=\{\overline{x}\}\notin \beta$ which is a contradiction. \\
	
	In order to prove the second part of Proposition~\ref{sequence property (S)}, we construct a sequence which is near to a subsequence of $(x_n)$.
	Since $A$ is a bounded set, up to a subsequence that we still denote by $(x_n)_n$, we can assume that $(\|x_n\|)$ is convergent to some $\alpha\geq 0$.
	Thanks to the first part of Proposition~\ref{sequence property (S)}, we have that $\alpha >0$. 	
	Let us redefine $\delta:= \min\{\alpha, \delta\}$.
	Now, we can further assume, up to a subsequence that we still denote by $(x_n)_n$, that $\|x_n\| \in [\alpha -\frac{\delta}{4}, \alpha + \frac{\delta}{4}]$.
	For each $n\in \N$, let us consider 
	\[x'_n :=  \left( \alpha -\frac{\delta}{4} \right) \frac{x_n}{\|x_n\|}.\]
	
	Notice that $\|x'_n\| =\alpha -\frac{\delta}{4}$ and that $\|x_n-x'_n\| \leq \delta/2$ for all $n\in \N$ . 
	Moreover, since $A$ is a semi-balanced set, we have that $(x'_n)_n\subset A$. 
	Finally, thanks to the triangle inequality, the sequence $(x'_n)_n\subset A$ and $\delta':= \delta/2>0$ can be chosen as the witnesses of property $(S)$ for the set $A$.	
\end{proof}
\begin{rem}
 In view of Proposition~\ref{sequence property (S)}, if $\beta$ is a bornology on $X$ satisfying property $(S)$, then $\beta$ must contain the relatively compact sets of $X$.
\end{rem}
In the following section we show that the Hadamard, weakly-Hadamard and limited bornologies are vector bornologies satisfying property $(S)$. 
Moreover, the Fr\'echet bornology trivially satisfies the property $(S)$.

\section{Property $(S)$ and the construction of a Lipschitz function}\label{preliminar results}

We start this section by showing that some well known bornologies satisfy the property~$(S)$ (see Definition~\ref{property S}). 
In this section $X$ denotes an infinite dimensional Banach space. 
Otherwise, the following three proposition are reduced to study the Fr\'echet bornology, and therefore, there is nothing to prove. 
Albeit simple, we present the proof of Proposition~\ref{compact S} and Proposition~\ref{limited S}. 

\begin{prop}\label{compact S}
The Hadamard bornology on $X$ satisfies property $(S)$.
\end{prop}
\begin{proof}
Let $A\subset X$ be a bounded non-relatively compact set. 
Then there exists a sequence $(x_n)_n\subset A$ with no accumulation points. 
Up to a subsequence, which we denote by $(x_n)_n$, we find $\sigma>0$ such that $\|x_n-x_m\|\geq \sigma$ for all $n\neq m$. 
Therefore, if we choose $\delta=\sigma/4$, then any sequence $(y_k)_k$ satisfying $\|y_k-x_{n_k}\|\leq \delta$, for some increasing sequence $(n_k)$, has no accumulation points.
Thus, the set $\{y_k:k\in\N\}$ is not relatively compact.
\end{proof}
\begin{prop}\label{limited S}
The Limited bornology on $X$  satisfies property $(S)$.
\end{prop}
\begin{proof}
Let $A\subset X$ be a bounded non-limited set. 
Then there is a weak${}^*$-null sequence $(x_n^*)_n\subset B_{X^*}$ which does not converge to $0$ uniformly on $A$. 
Hence, up to a subsequence of $(x_n^*)_n$, there exist a sequence $(x_n)_n\subset A$ and $\sigma >0$ such that $|x^*_n(x_n)| \geq \sigma$ for all $n\in \N$. 
Considering $\delta=\sigma/2$, we obtain that any vector $y_k\in B(x_{n_k},\delta)$, with $(n_k)_k$ an increasing sequence, satisfies $|x^*_{n_k}(y_k)|\geq \sigma/2$.
Therefore, the set $\{y_k:k\in\N\}$ is not limited.
\end{proof}
The following result concerns the weak-Hadamard bornology. It can be found inside of the proof of \cite[Theorem 2.1]{BFV}.
\begin{prop}\label{weak compact S}
The weak-Hadamard bornology on $X$ satisfies property $(S)$.
\end{prop}
\begin{proof}

Let $A\subset X$ be a bounded non-relatively weakly-compact set. 
By the Eberlein-\u{S}mulian Theorem, there is a sequence $(x_n)_n\subset X$ with no weakly-convergent subsequence. 
By contradiction, suppose that no subsequence of $(x_n)$ satisfies the statement of property $(S)$. 
Then, there exist an increasing sequence $(n(1,j))_j\subset \N$ and a sequence $(z^1_{n(1,j)})_j$ weakly-convergent to $z^1$ such that $z^1_{n(1,j)}\in B(x_{n(1,j)},1)$ for all $j\in \N$. 
Inductively, for $k\geq 2$, there exist a subsequence $(n(k,j))_j$ of $(n(k-1,j))_j$ and $(z^k_{n(k,j)})_j$ weakly-convergent sequence to $z^k$ such that $z^k_{n(k,j)}\in B(x_{n(k,j)},1/k)$ for all $j\in \N$. 
Let us show that the sequence $(z^k)_k$ converges in norm to some $z^\infty\in X$. 
Indeed, let $k<l$. 
Recalling that the norm is a weakly-lower semi continuous and that $(n(l,j))_j$ is a subsequence of $(n({k,j}))_j$, we obtain
\[\|z^k-z^l\|\leq \liminf_j \|z^k_{n({l,j})}-z^l_{n({l,j})}\| \leq \liminf_j \|z^k_{n({l,j})}-x_{n({l,j})}\| +\|x_{n({l,j})}-z^l_{n({l,j})}\|\leq \dfrac{1}{k}+\dfrac{1}{l},\]
\noindent proving that $(z^k)_k$ is a norm-Cauchy sequence. 
We claim that $(x_{n(k,k)})_k$ weakly-converges to $z^\infty$. 
Let $x^*\in S_{X^*}$ and $\varepsilon>0$. 
Let $n_0 \in \N$ such that $n_0^{-1}\leq \varepsilon/3$.
Thus, $\|z^k-z^\infty\|\leq \varepsilon/3$ for all $k\geq n_0$. 
Since the sequence $(z^{n_0}_{n(n_0,j)})_j$ is weakly-convergent to $z^{n_0}$, there is $m_0\in \N$ such that $|\langle x^*, z^{n_0}-z^{n_0}_{n({n_0,j})}\rangle| \leq \varepsilon/3$ for all $j\geq m_0$. 
Observe that, for any $k> n_0$, there is $j_k\in \N$ such that $n(k,k)=n(n_0,j_k)$. 
Hence, for $k$ large, we have that $j_k>m_0$ and then

\begin{align*}
	|\langle x^*,x_{n(k,k)}-z^\infty\rangle | &\leq |\langle x^*,x_{n(n_0,j_k)}-z^{n_0}_{n(n_0,j_k)}\rangle |+|\langle x^*,z^{n_0}_{n(n_0,j_k)}-z^{n_0}\rangle |+|\langle x^*,z^{n_0}-z^\infty\rangle |\\
	&\leq n_0^{-1} + \dfrac{\varepsilon}{3} + \|z^{n_0}-z^\infty\| \leq \varepsilon,\\
\end{align*}
\noindent concluding that the sequence $(x_{n(k,k)})_k$ weakly-converges to $z^\infty$. 
Therefore, the set $\{x_{n(k,k)}:k\in\N\}$ is weakly-compact, which is a contradiction. 

\end{proof}

The last ingredient of the proof of Theorem~\ref{beta differentiability} is the Lipschitz function constructed in Lemma~\ref{beta lipschitz} below.
A similar construction is used to prove certain properties of non-reflexive spaces, see \cite{BFV}. 
In order to continue, let us give some definitions and state some simple results that can be found in \cite{BFT}.\\

For a set $A\subset X$, $\textup{cone}(A)$ denotes the set $\{\lambda x: x\in A,~\lambda \geq 0\}$.

\begin{defn}
	Let $X$ be a Banach space, let $(x_n)_n\subset X$ such that $\|x_n\|=\|x_m\|$ for all $n,m \in \N$ and let $\sigma\in (0,\|x_1\|)$. 
	\begin{enumerate}
		\item We say that $(x_n)_n$ is $\sigma$-separated if $\|x_n-x_m\|\geq \sigma$ for all $n,m\in \N$, with $n\neq m$.
		\item We say that $(x_n)_n$ is $\sigma$-cone separated if the sets $\{\textup{cone}(B(x_n,\sigma))\setminus \{0\}:n\in \N\}$ are pairwise disjoint.
	\end{enumerate}
\end{defn}
By definition, a $\sigma$-cone separated sequence is $2\sigma$-separated. 
Reciprocally, we have the following result.
\begin{prop}{\cite[Proposition 3.6]{BFT}}\label{sigma/4}
	Let $(x_n)_n\subset X$ be a $\sigma$-separated sequence. Then $(x_n)_n$ is $\sigma/4$-cone separated.
\end{prop}
The core of a set $A\subset X$, denoted by $\textup{core}(A)\subset X$, is the set defined by
\[\textup{core}(A):=\{x\in A: \forall y\in S_X,\exists t>0,~ [x,x+ty)\subset A\}.\]
Let $f: X\to \R$ be a function. We denote $\{f=f(x)\}:=\{z \in X: f(z)=f(x)\}$.
\begin{prop}{\cite[Proposition 3.2]{BFT}}\label{coregateaux}
Let $f: X\to \R$ be a function. If $x\in \textup{core}(\{f=f(x)\})$, then $f$ is G\^ateaux-differentiable at $x$ with G\^ateaux-differential equal to $d_Gf(x)=0$.
\end{prop}

Finally, we proceed with the construction of the mentioned Lipschitz function.
\begin{lem}\label{beta lipschitz}
	Let $\beta$ be a vector bornology on $X$, different from the Fr\'echet bornology, satisfying property $(S)$. 
	Let $A\subset X$ be semi-balanced set such that $A\notin \beta$.
	Then, there exist $\sigma>0$ and a $\sigma$-separated sequence $(x_n)_n\subset A$ such that the Lipschitz function ${f:X\to \R}$ defined by 
	
	\[f(x):= \textup{dist}\left( x, X\setminus \bigcup_{n=1}^\infty B\left(\dfrac{x_n}{n},\dfrac{\sigma}{4n}\right)\right),~\text{for all }x\in X,\]
	is $\beta$-differentiable at $0$ but not Fr\'echet-differentiable at $0$.
\end{lem}

\begin{proof}
	Let $(x_n)_n\subset A$ and let $\delta>0$ given by property $(S)$. 
	Thanks to Proposition~\ref{sequence property (S)}, we assume that $\| x_n\| =\alpha$, for some $\alpha>0$, for all $n\in \N$.
	Since the sequence $(x_n)_n$ do not have accumulation points, up to a subsequence, we can assume that $(x_n)_n$ is a $\sigma$-separated sequence, for some $0<\sigma \leq \delta$. 
	Let $f:X\to \R$ be the $1$-Lipschitz function on $X$ defined by
	\[f(x):= \textup{dist}\left(x, X\setminus \bigcup_{n=1}^\infty B\left( \dfrac{x_n}{n},\dfrac{\sigma}{4n}\right)\right),~\text{for all }x\in X.\]
	By Proposition~\ref{sigma/4} and Proposition~\ref{coregateaux}, $f$ is G\^ateaux-differentiable at $0$, with G\^ateaux-differential equal to $d_Gf(0)=0$. 
	However, since $nf(x_n/n)=\sigma/4$, $f$ is not Fr\'echet-differentiable at $0$. 
	Finally, it only remains to prove that $f$ is $\beta$-differentiable at $0$. 
	We proceed by contradiction. 
	Suppose, for some set $W\in \beta$, the ratio of differentiability do not converge uniformly on $W$. 
	That is, there exist a null sequence $(t_k)\subset \R^+$, $(w_k)_k\subset W$ and $\varepsilon>0$ such that:
	
	\[\left|\dfrac{f(t_kw_k)}{t_k}\right|\geq \varepsilon, ~\forall n\in \N.\]
	
	Since $f(t_kw_k)>0$, there is a sequence $(n_k)\subset \N$ such that $t_kw_k\in B(x_{n_k}/n_k, \sigma/4n_k)$. 
	Due to the fact that $W$ is bounded and that $(t_kw_k)_k$ converges to $0$, we can assume, up to a subsequence, that $(n_k)_k$ is increasing. 
	We have two different cases now. 
	If the sequence $(w_k)_k$ tends to $0$, the set $\{w_k:~k\in \N\}$ is relatively compact.
	However, the quotient of differentiability at a point of G\^ateaux differentiability converges uniformly on relatively compact sets for Lipschitz functions.
	This contradicts the fact that $\varepsilon>0$. 
	Therefore, we can assume that $(w_k)_k$ is not a norm-null sequence and then, up to a subsequence, the sequence $(\|w_k\|)_k$ converges to some $\nu>0$.
	Since $t_kw_k\in B(x_{n_k}/n_k, \sigma/4n_k)$, then $n_kt_kw_k\in B(x_{n_k},\sigma/4)$, Therefore,
	\[ n_kt_k \in \left[ \dfrac{\alpha}{\|w_k\|}-\dfrac{\sigma}{4\|w_k\|},  \dfrac{\alpha}{\|w_k\|}+\dfrac{\sigma}{4\|w_k\|} \right]. \] 
	Thus, the sequence $ (t_kn_k)_k$ accumulates in $[\frac{\alpha}{\nu}-\frac{\sigma}{4\nu},\frac{\alpha}{\nu}+\frac{\sigma}{4\nu}]$. 
	Passing through a subsequence, we assume that $(t_kn_k)_k$ converges to some $\lambda>0$. 
	Hence, there is $K\in \N$ such that $\lambda w_k\in B(x_{n_k},\sigma)$, for all $k\geq K$. This is a contradiction with the property $(S)$ for the bornology $\beta$ because $(\lambda w_k)_k\subset \lambda W\in \beta$ and $\|\lambda w_k-x_{n_k}\|\leq \sigma\leq \delta$ for all $k\geq K$.
\end{proof}

Let us end this section with the following result, which can be seen as a natural way to construct vector bornologies that satisfies the property $(S)$.
\begin{prop}
Let $X$ and $Z$ be two Banach spaces and let $f:X\to Z$ be a Lipschitz function. Assume that $f$ is G\^ateaux differentiable at $0$. Then, the family $\mathcal{F}$ of nonempty bounded sets $A\subset X$ such that, for every $a\in X$ and every $\lambda \in \R$
\[\lim_{t\to 0} \sup_{x\in a+\lambda A}\left \|\dfrac{f(0+tx)-f(0)}{t}-d_Gf(0)(x)\right\|=0,\]
is a vector  bornology on $X$ satisfying the property $(S)$.
\end{prop}

\begin{proof}
Without loss of generality, we assume that $f(0)=0$ and $d_Gf(0)=0$. Let us first prove that $\mathcal{F}$ is a vector bornology. 
Since $f$ is Gateaux differentiable at $0$, we know that $\{x\}\in \mathcal{F}$ for all $x\in X$. 
Therefore, $\mathcal{F}$ covers $X$.
The other two properties of bornologies are trivially satisfied by $\mathcal{F}$ thanks to the algebra of limits. On the other hand, it is clear from the above limit, that if $a\in X$, $\lambda \in \R$ and $A\in \mathcal{F}$, then, $a+\lambda A\in \mathcal{F}$. It remains to show that if $A\in \mathcal{F}$, then $\textnormal{bal}(A)\in \mathcal{F}$. Indeed, since $[-1,1]$ is compact, then for every $N\in \N\setminus \lbrace 0 \rbrace$, there exists $n_N\in \N$ and $(\beta_i)_{1\leq i\leq n_N}\subset [-1,1]$ such that $[-1,1]=\cup_{i=1}^{n_N} [\beta_i-1/N,\beta_i+1/N]$. Fix $a\in X$ and $\lambda\in \R$, 
\begin{eqnarray*}
\sup_{x\in a+\lambda \textnormal{bal}(A)}\left \|\dfrac{f(tx)}{t}\right\|&=&\sup_{x\in \textnormal{bal}(A)}\left \|\dfrac{f(t(a+\lambda x))}{t}\right\|\\
                                                                                                                   &=& \sup_{\beta\in [-1,1]} \sup_{x\in A}\left \|\dfrac{f(t(a+\lambda \beta x))}{t}\right\|\\
                                                                                                                   &=& \max_{1\leq i\leq n_N} \sup_{\beta\in [\beta_i-1/N,\beta_i+1/N]} \sup_{x\in A}\left \|\dfrac{f(t(a+\lambda \beta x))}{t}\right\|\\
                                                                                                                    &\leq&  \max_{1\leq i\leq n_N} \sup_{\beta\in [\beta_i-1/N,\beta_i+1/N]} \sup_{x\in A}\left \|\dfrac{f(t(a+\lambda \beta_i x))}{t}\right\|\\
                                                                                                                     & & + \max_{1\leq i\leq n_N} \sup_{\beta\in [\beta_i-1/N,\beta_i+1/N]} \sup_{x\in A} \left \|\dfrac{f(t(a+\lambda \beta_i x))}{t}-\dfrac{f(t(a+\lambda \beta x))}{t}\right\|\\
                                                                                                                   &\leq& \max_{1\leq i\leq n_N}\sup_{x\in A}\left \|\dfrac{f(t(a+\lambda \beta_i x))}{t}\right\| \\
                                                                                                                   & & +\max_{1\leq i\leq n_N} \sup_{\beta\in [\beta_i-1/N,\beta_i+1/N]} \sup_{x\in A}\textup{Lip}(f)|\lambda||\beta-\beta_i|\|x\|\\
                                                                                                                    &\leq& \max_{1\leq i\leq n_N}\sup_{x\in A}\left \|\dfrac{f(t(a+\lambda \beta_i x))}{t}\right\| +\frac{|\lambda|\textup{Lip}(f) \sup_{x\in A}\|x\|}{N}\\
                                                                                                                    &=& \max_{1\leq i\leq n_N}\sup_{x\in a +\lambda \beta_i A}\left \|\dfrac{f(tx)}{t}\right\| +\frac{|\lambda|\textup{Lip}(f) \sup_{x\in A}\|x\|}{N}.
\end{eqnarray*}
Since $A\in \mathcal{F}$, tending $t$ to $0$, we get that, for every $N\in \N\setminus \lbrace 0 \rbrace$
\begin{eqnarray*}
\limsup_{t\to 0} \sup_{x\in a+\lambda \textnormal{bal}(A)}\left \|\dfrac{f(tx)}{t}\right\| \leq \frac{|\lambda|\textup{Lip}(f) \sup_{x\in A}\|x\|}{N}.
\end{eqnarray*}
Hence, 
\begin{eqnarray*}
\lim_{t\to 0} \sup_{x\in a+\lambda \textnormal{bal}(A)}\left \|\dfrac{f(tx)}{t}\right\| =0.
\end{eqnarray*}
Thus, $\textnormal{bal}(A)\in \mathcal{F}$ and finally, we have that $\mathcal{F}$ is a vector bornology. \\

Now, let us see that $\mathcal{F}$ satisfies the property $(S)$. Let $A\subset X$ be a bounded set such that $A\notin \mathcal{F}$. 
Therefore there are  $a_0\in X$ and $\lambda_0\in \R\setminus \lbrace 0 \rbrace$,  such that 

\[\lim_{t\to 0} \sup_{x\in a_0+\lambda_0 A}\left \|\dfrac{f(tx)}{t}\right\|>0.\]
Thus, there are $\sigma >0$, $(x_n)_n\subset A$ and $(t_n)_n\subset \R$, convergent to $0$,  such that

\[\left \| \dfrac{f(t_n (a_0+ \lambda_0 x_n))}{t_n}\right \| \geq  \sigma,~ \text{ for all }n\in\N. \]

We claim that the sequence $(x_n)_n$ and $\delta:= \sigma/2|\lambda_0|\textup{Lip}(f)$ witness the property $(S)$ for the set $A$.
Indeed, if $(n_k)_k\subset \N$ is an increasing sequence and $(y_k)_k\subset X$ is a sequence such that $\|x_{n_k}-y_k\|\leq \delta$ equivalently, $\|(a_0+ \lambda_0 x_{n_k})-(a_0+ \lambda_0 y_k)\|\leq |\lambda_0|\delta$, then

\begin{eqnarray*}
\left\| \dfrac{f(t_{n_k} (a_0+ \lambda_0  y_k))}{t_{n_k}}\right \| &\geq& \left \| \dfrac{f(t_{n_k} (a_0+ \lambda_0  x_{n_k}))}{t_{n_k}}\right \| -\left\| \dfrac{f(t_{n_k} (a_0+ \lambda_0  y_k))-f(t_{n_k} (a_0+ \lambda_0  x_{n_k}))}{t_{n_k}}\right \|\\
&\geq& \sigma -\textup{Lip}(f)\dfrac{|\lambda_0| \delta t_{n_k}}{t_{n_k}} \geq \dfrac{\sigma}{2}.
\end{eqnarray*}
Therefore  $\{y_k:k\in\N\}\notin \mathcal{F}$. This ends the proof.
\end{proof}

\section{The main result: Characterization of $\beta_Y$-$\beta_X$-operators.}\label{main result}
In order to prove the main result of this section, we need the following definitions.

\begin{defn}
	Let $X$ and $Y$ be two Banach spaces and let $\beta_X$ and $\beta_Y$ be vector bornologies on $X$ and $Y$ respectively. 
	Let $T: Y\to X$ be any operator (not necessarily linear). 
	We say that $T$ is a $\beta_Y$-$\beta_X$-operator, if  $T(\beta_Y)\subset \beta_X$, that is
	\[T (A) \in \beta_X,~\hspace{0.5cm} \text{for all } A \in \beta_Y.\]

Whenever $\beta_Y$ is the Fr\'echet bornology, we simple write that $T$ is a $\beta_X$-operator.
\end{defn}
 Recall that we denote by $\beta=G, H, L, wH$ and $F$, the G\^ateaux, Hadamard, limited, weak Hadamard and Fr\'echet bornologies respectively. 
 The above definition unifies some classical notions of linear operators, which we recall in the following example.

\begin{ex} Let $X$ and $Y$ be two Banach spaces. Then.

$(i)$ Linear $F$-$H$-operators or $H$-operators coincide with  linear compact operators from $Y$ into $X$.

$(ii)$ Linear $F$-$L$-operators or  $L$-operators coincide with linear limited operators from $Y$ into $X$.

$(iii)$  Linear $F$-$wH$-operators or $wH$-operators coincide with linear weakly compact operators from $Y$ into $X$.

$(iv)$  Linear $wH$-$H$-operators coincide with completely continuous linear operators from $Y$ into $X$.
\end{ex}
We need the following notion of generalized differential in the spirit of Suchomolinov (see for instance \cite[p. 135]{Na} and \cite{MN})  which is a generalization of $\beta$-differentiability.
\begin{defn}\label{ph differentiability}
	Let $X$ and $Y$ be Banach spaces and let $\beta$ be a bornology on $Y$. A map ${T:Y\to X}$ is said to be $\beta$-positively homogeneous differentiable at $y\in Y$  ($\beta$-p.h. differentiable for short) if there is a continuous positively homogeneous function $R:Y\to X$ such that
	
	\[ \lim_{t\to 0^+} \sup_{z\in A} \|\dfrac{T(y+tz)-T(y)}{t}- R(z) \|=0,~ \text{for all }A\in \beta. \]  
	We denote the $\beta$-p.h. differential of $T$ at $y$,  by $R:=d_{\beta}^{ph} T(y)$.
\end{defn}

The following remark summarizes some simple facts about the p.h.-differentiability. In particular, we use the second and fourth points of Remark~\ref{homogeneous} without referring to them.

\begin{rem} \label{homogeneous}

 Let $T:Y\to X$ be an operator (not necessarily linear) and let $\beta$ be a bornology on $Y$. Then: 
 \begin{itemize}
 	\item  If $T$ is $\beta$-p.h. differentiable at $y\in Y$, then the $\beta$-p.h. differential of $T$ at $y$ is unique.
 	\item Definition~\ref{ph differentiability} extends the usual notion of differentiability. 
 	Indeed, if $T$ is $\beta$-differentiable at some $y\in Y$, then $d_\beta T(y)= d_\beta ^{ph} T(y)$. In particular, if $T$ is a linear bounded operator, then $ d_\beta ^{ph} T(y)=d_\beta T(y)=T$, for every $y\in Y$ and any $\beta$ bornology on $Y$.
 	\item If $T$ is a continuous positively homogeneous function, then $T$ is $\beta$-p.h. differentiable at $0$ for any $\beta$ bornology on $Y$ and we have that $d_\beta^{ph} T(0)= T$.
 	\item If $T$ is $\beta$-p.h. differentiable at $y\in Y$, then $d_\beta^{ph}T(y)$ sends bounded sets into bounded sets.
     \item Notice that  every $\beta$-p.h. differentiable operator $T:Y \to X$ such that  $d_{\beta} ^{ph} T(0)=0(=d_{\beta} T(0))$ is such that $f\circ T$ is $\beta$-differentiable, for every Lipschitz function $f : X\to \R$ (not necessarily differentiable). Indeed, we have that 
\[\lim_{t\to 0}\sup_{u\in A}\left| \dfrac{f\circ T(tu)-f\circ T(0)}{t}\right| \leq \textup{Lip}(f) \lim_{t\to 0}\sup_{u\in A}\left\| \dfrac{T(tu)-T(0)}{t} \right\| =0, ~ \text{for any }A\in \beta.\]
 \end{itemize}
\end{rem}

\vskip5mm

Let $X$ and $Y$ be two Banach spaces and let $\beta_X$ and $\beta_Y$ be vector bornologies on $X$ and $Y$ respectively. Then, if  $R: Y\to X$ is a continuous positively homogeneous such that there is $A\in \beta_Y$ satisfying $R(A)\not \in \beta_X$ (that is, $R$ is not $\beta_Y$-$\beta_X$-operator), then, there exists a semi-balanced set $B\in \beta_Y$ such that $R(B)$ is a  semi-balanced, bounded set  that does not belong to $\beta_X$. 
Indeed, thanks to the fact that $R$ is continuous and positively homogeneous, it is enough to consider $B=\textnormal{bal}(A)$, the balanced hull of $A$. 
Now, we are able to state the main result of this work. 
\begin{thm}\label{beta differentiability}
	Let $X$ and $Y$ be two Banach spaces and let $\beta_X$ and $\beta_Y$ be vector bornologies on $X$ and $Y$ respectively.
	Assume that $\beta_X$ satisfies property $(S)$ and let $T: Y\to X$ be a $\beta_Y$-p.h. differentiable operator at $y\in Y$. 
	Then the following assertions are equivalent.

$(i)$ $d_{\beta_Y} ^{ph} T(y)$ is a $\beta_Y$-$\beta_X$-operator

$(ii)$   for every Lipschitz function $f:X\to \R$, $\beta_X$-p.h. differentiable at $x=Ty$, we have that $f\circ T$ is $\beta_Y$-p.h. differentiable at $y$ and $d_{\beta_Y}^{ph} (f\circ T)(y)=d_{\beta_X}^{ph}f(x)\circ  d_{\beta_Y}^{ph}T(y)$

$(iii)$  for every Lipschitz function $f:X\to \R$, $\beta_X$-differentiable at $x=Ty$ with $d_{\beta_X} f(x)=0$, we have that $f\circ T$ is $\beta_Y$-differentiable at $y$ and $d_{\beta_Y}(f\circ T)(y)=0$.
\end{thm}

Let us start with the proof of Theorem \ref{weak hadamard differentiability} and Theorem \ref{Gateaux-weak hadamard differentiability}.

\begin{proof}[Proof of Theorem~\ref{weak hadamard differentiability}]
	It is a direct application of Theorem~\ref{beta differentiability} where $\beta_X$ is the weakly-Hadamard bornology, $\beta_Y$ is the Fr\'echet bornology. 
	Indeed, Proposition~\ref{weak compact S} asserts that the weakly-Hadamard bornology satisfies property $(S)$.
\end{proof}

\begin{proof}[Proof of Theorem~\ref{Gateaux-weak hadamard differentiability}]
	It is a direct application of Theorem~\ref{beta differentiability} where $\beta_X$ is the Hadamard bornology, $\beta_Y$ is the weak-Hadamard bornology. 
	Indeed, Proposition~\ref{compact S} asserts that the Hadamard bornology satisfies property $(S)$. 
\end{proof}

Also, as a direct consequence of Theorem~\ref{beta differentiability}, using the Fr\'echet bornology on $Y$ and the Hadamard (resp. limited) bornology on $X$, we characterize compact operators (resp. limited operators). 
Thus, we recover the main result of \cite{BFT}. 

\begin{proof}[Proof of Theorem~\ref{beta differentiability}]
$(i)\Longrightarrow (ii)$. This part is straightforward and it does not require property $(S)$.
	Indeed, let $T:Y\to X$ be a $\beta_Y$-p.h. differentiable at $y$, with $\beta_Y$-p.h. differential equal to $R:=d_ {\beta_Y}^{ph}T(y)$, which is a $\beta_Y$-$\beta_X$ operator at $y$ and let $f:Y\to \R$ be a Lipschitz function $\beta_X$-p.h. differentiable at $x=Ty$, with $p:=d_{\beta_X}^{ph} f(x)$. 
	We claim that the $\beta_Y$-p.h. differential of $f\circ T$ at $x$ is equal to $p\circ R$. 
	Indeed, for any $A\in \beta_Y$, we have that for all $u\in A$ and all $t>0$
	
	\begin{align*}
		\sup_{u\in A}\left | \dfrac{f\circ T (y+tu)-f\circ T(y)}{t} - p\circ R(u) \right | &\leq  \sup_{u\in A}\left | \dfrac{f\circ T (x+tu)-   f(T(y)+tR(u))}{t}\right| \\
&\hspace{0.5cm} + \sup_{u\in A}\left|\dfrac{f(T(y)+tR(u))  -   f(T(y))}{t}-p(R(u))  \right |\\
		&\leq \sup_{u\in A}  \textup{Lip}(f) \left \| \dfrac{T(y+tu) -T(y)-tR(u)}{t}  \right \|\\
&\hspace{0.5cm}+ \sup_{v\in R(A)}\left | \dfrac{f(x+tv)-f(x)}{t} -p(v)\right |.
	\end{align*}
Therefore, thanks to the $\beta_Y$-p.h. differentiability of $T$, $R(A)\in \beta_X$ and the $\beta_X$-p.h. differentiability of $f$, sending $t$ to $0$ in the above expression, we obtain that $f\circ T$ is $\beta_Y$-p.h. differentiable at $y$ with $\beta_Y$-p.h. differential  equal to $p\circ R$.\\

$(ii)\Longrightarrow (iii)$. This part is trivial. Notice just that since $d_{\beta_Y}^{ph} (f\circ T)(y)=0$, then $f\circ T$ is necessarily $d_{\beta_Y}$-differentiable at $y$ and  $d_{\beta_Y} (f\circ T)(y)=d_{\beta_Y}^{ph} (f\circ T)(y)=0$.\\

$(iii)\Longrightarrow (i)$. Redefining $T$ by $T:=T(\cdot+y)-T(y)$, we assume without loss of generality that $y=0$ and $T(0)=0$. We proceed by contradiction.  Assume that $R:=d_{\beta_Y} ^{ph} T(0)$ is not a $\beta_Y$-$\beta_X$-operator. 
Therefore, there is a balanced set $A\in \beta_Y$ such that $R(A)$ is a bounded, semi-balanced set and $R(A)\notin \beta_X$. 
Since $\beta_X$ satisfies property $(S)$, Lemma~\ref{beta lipschitz} gives us a $\sigma>0$, a $\sigma$-separated sequence $(x_n)_n\subset R(A)$ and the $1$ Lipschitz function $f:X\to \R$ defined by 
\[ f(x):= \textup{dist} \left(x,~X\setminus \bigcup_{n=1}^\infty B \left(\dfrac{x_n}{n},\dfrac{\sigma}{4n}\right) \right), ~\text{for all }x\in X,\]

which is $\beta_X$-differentiable (and thus $\beta_X$-p.h. differentiable) at $0$ with $d_{\beta_X}^{ph} f(0)=0$. 
Let us see that $f\circ T$  is not $\beta_Y$-differentiable at $0$ neither.
Indeed, let $(y_n)_n\subset A$ such that $R(y_n)=x_n$ for all $n\in \N$.
Finally, notice that (since $f\circ T(0)=0$ and $f$ is positive)

\begin{align*}
	\left|\dfrac{f\circ T (\frac{y_n}{n})-f\circ T (0) }{\frac{1}{n}}\right|&=\dfrac{f\circ T (\frac{y_n}{n})-f\circ T (0) }{\frac{1}{n}}\\
          &=\dfrac{f\circ T (\frac{y_n}{n})-   f (\frac{R(y_n)}{n})}{ \frac{1}{n}}  +
	 \dfrac{f ( \frac{R(y_n)}{n})- f\circ T (0) }{\frac{1}{n}}\\
	&\geq -\textup{Lip}(f) \left\| \dfrac{T(\frac{y_n}{n})-\frac{R(y_n)}{n}}{\frac{1}{n}} \right\| + nf(\frac{x_n}{n})\\
	&= \dfrac{\sigma}{4}-\textup{Lip}(f) \left\| \dfrac{T(\frac{y_n}{n})}{\frac{1}{n}} -R(y_n)\right\|,~\text{for all }n\in \N.
\end{align*}
 By statement $(iii)$, we have that $d_{\beta_Y} f\circ T (0)=0$. 
 However, sending $n$ to infinity in the above expression and using the $\beta_Y$-p.h. differentiability of $T$, we conclude that $0$ is not the $\beta_Y$ differential $f\circ T$ at $0$, which is a contradiction.
\end{proof}

It is a well-known result that  the Fr\'echet derivative of a (not necessarily linear) compact function (resp. completely continuous function)  is  a linear compact (resp. a linear completely continuous) operator. For more information and details, we refer to  \cite{Da, Kr, MePe, Mo, Nu, MN}). We give in the following theorem a generalization of these results, which applies to the more general notions of $\beta$-p.h. differentiability and  the notion of $\beta_Y$-$\beta_X$-operator. We need the following lemma.
\begin{lem} \label{lem1} 
	Let $X$ be a Banach space and let $\beta$ be a vector bornology on $X$ satisfying property $(S)$. Let $A\subset X$ and assume that for every $\varepsilon>0$ there exists $A_\varepsilon\in \beta$ such that $A\subset A_\varepsilon+\varepsilon B_X$. Then $A \in \beta$.
	\end{lem}

\begin{proof}  
Suppose by contradiction that $A\not \in \beta$. By property $(S)$, there are a sequence $(x_n)_n\subset A$ and $\delta >0$ such that for any increasing sequences $(n_k)\subset \N$ and for any sequence $(y_k)_k\subset X$ satisfying $\|y_k-x_{n_k}\| \leq \delta$ for all $k\in\N$, the set $\{y_k:k\in\N\}\notin \beta$. By assumption, with $\varepsilon=\delta$, we have that for every $k\in \N$, there exists $z_k\in A_\delta$ such that $\|z_k-x_{n_k}\| < \delta$. This contradict the property $(S)$, since $\{z_k:k\in\N\}\subset A_\delta \in \beta$.
\end{proof}

\begin{thm} \label{theorem-beta-beta-op} Let $X$ and $Y$ be two Banach spaces and let $\beta_X$ and $\beta_Y$ be vector bornologies on $X$ and $Y$ respectively. Assume that $\beta_X$ satisfies property $(S)$. Let $T: Y\to X$ be a $\beta_Y$-$\beta_X$-operator which is $\beta_Y$-p.h. differentiable  at $y\in Y$. Then, $d_{\beta_Y}^{ph} T(y)$ is a $\beta_Y$-$\beta_X$-operator.
\end{thm}

\begin{proof} Let $A\in \beta_Y$. From the definition of the $\beta_Y$-p.h. differentiability  of $T$ at $y\in Y$, we have that for every $\varepsilon >0$, there exists $\delta_\varepsilon>0$ such that 
\[  \sup_{z\in A} \|\delta_\varepsilon^{-1}(T(y+\delta_\varepsilon z)-T(y))- d_{\beta_Y}^{ph} T(y)(z) \|\leq \varepsilon. \]  
This shows that, for every $\varepsilon >0$, there exists $\delta_\varepsilon>0$ such that 
\[d_{\beta_Y}^{ph} T(y)(A)\subset \delta_\varepsilon^{-1}(T(y+\delta_\varepsilon A)-T(y)) +\varepsilon B_X.\]
Since $T$ is $\beta_Y$-$\beta_X$-operator and $\beta_X$ is a vector bornology, it follows that $$\delta_\varepsilon^{-1}(T(y+\delta_\varepsilon A)-T(y))\in \beta_X.$$ Thus, using Lemma~\ref{lem1}, we get that $d_{\beta_Y}^{ph} T(y)(A)\in \beta_X$. Hence, $d_{\beta_Y}^{ph} T(y)$ is a $\beta_Y$-$\beta_X$-operator. 
\end{proof}



\section{Alternative proof of Theorem~\ref{limited differentiability} }\label{limited section} 
While G\^ateaux-differentiability and limited-differentiability do not generally coincide outside the Gelfand-Phillips spaces for general Lipschitz functions (see Corollary \ref{theorem0}), they always coincide for convex continuous functions on any Banach space as we show in Proposition \ref{convex differentiability} below. Let us provide an alternative proof of Theorem~\ref{limited differentiability} which was initially given in \cite{B}.
We recall the following two results that can be found in \cite{BV}.

\begin{prop}\cite[Proposition 8.1.1]{BV}\label{convex function}
	Let $\beta$ be a bornology on $X$. 
	Let $(x_n^*)_n\subset X^*$ be a bounded sequence. 
	Let $f:X\to \R$ be the convex function defined by:
	\[f(x)=\sup_n\left\{0,x^*_n(x)-\dfrac{1}{n}\right\}.\]
	Then, $f$ is $\beta$-differentiable at $0$ if and only if $(x^*_n)_n\to_{\tau_\beta} 0$, where $\tau_\beta$ denotes the topology of the uniform convergence on $\beta$-sets on $X^*$. 
\end{prop}
\begin{thm}{\cite[Theorem 8.1.3]{BV}}\label{beta12}
	Let $X$ be a Banach space and let $\beta_1$ and $\beta_2$ be two bornologies on $X$ such that $\beta_1\subseteq \beta_2$. 
	Then, the following assertions are equivalent.
	\begin{enumerate}
		\item[(a)]$\tau_{\beta_1}$ and $\tau_{\beta_2}$ agree sequentially in $X^*$.
		\item[(b)] $\beta_1$-differentiability and $\beta_2$-differentiability coincide for continuous convex functions.
	\end{enumerate}
\end{thm}

We know that G\^ateaux differentiability and Hadamard differentiability coincide for real-valued Lipschitz functions on any Banach space. The following proposition can be seen as the analogous of this result for convex functions.

\begin{prop}\label{convex differentiability}
	Let $X$ be a Banach space. G\^ateaux differentiability and limited differentiability coincide for continuous convex functions.
\end{prop}
\begin{proof} We use Theorem~\ref{beta12} with the bornologies $\beta_1$= G\^ateaux and $\beta_2$= limited. 
	Let $(x_n^*)_n\subset X^*$ be a sequence $\tau_{\beta_1}$-convergent to $0$, i.e., $(x_n^*)_n$ is a weak${}^*$-null sequence. 
	Let $A\subseteq X$ be limited set on $X$. 
	By definition of limited set, we have that
	\[\lim_{n\to\infty}\sup_{x\in A}|x^*_n(x)|=0. \]
	Since $A$ was an arbitrary limited set on $X$, we have that $(x^*_n)_n$ converges to $0$ in $\tau_{\beta_2}$. 
	Applying Theorem~\ref{beta12} we obtain the desired result.
\end{proof}

In \cite{B}, the proof of Theorem~\ref{limited differentiability} follows a completely different approach compared to our proof of Theorem~\ref{beta differentiability}. As a direct consequence of Proposition~\ref{convex differentiability}, we obtain a simplified proof of Theorem~\ref{limited differentiability}.

\begin{proof}[Alternative proof of Theorem~\ref{limited differentiability}]
	Thanks to Proposition~\ref{convex differentiability}, the necessity of Theorem~\ref{limited differentiability} is straightforward. 
	In fact, it is analogous to the necessity of Theorem~\ref{beta differentiability}. 
	Conversely, we proceed by contradiction. 
	Let $T\in \mathcal{L}(Y,X)$ be a non-limited operator. 
	Then, there exist a weak${}^*$-null sequence $(x_n^*)_n\subset X^*$ and a sequence $(y_n)_n\subset B_Y$ such that $x^*_n(Ty_n)\geq2$. 
	Let us consider the function $f:X\to\R$ defined by:
	\[f(x)=\max\left\{0,\sup\left\{x^*_n(x)-\dfrac{1}{n}\right\}\right\},~\text{for all } x\in X,\]
	which is G\^ateaux-differentiable at $0$, thanks to Proposition~\ref{convex function}. 
	Since $f$ is a positive function and $f(0)=0$, we know that $d_Gf(0)=0$. In fact, the only candidate to Fr\'echet-differential to $f\circ T$ at $0$ is also the functional $0$. 
	However, the computation 
	\[nf\circ T\left( \dfrac{y_n}{n}\right) \geq n\left( \dfrac{2}{n}-\dfrac{1}{n}\right) =1,\]
	shows that $f\circ T$ is not Fr\'echet-differentiable at $0$.
\end{proof}

\section{Consequences of Theorem~\ref{beta differentiability}}\label{consequences}
We give in this section some applications of our main result.
\subsection{Characterization of a $\beta_2$-$\beta_1$-space}
We start with the following definitions.
\begin{defn}
 Let $X$ be a Banach space and $\beta_1$, $\beta_2$ be two vector bornologies on $X$. We say that $X$ is a $\beta_2$-$\beta_1$-space if $\beta_2 \subset \beta_1$ or equivalently, if the identity $Id:X \to X$ is a $\beta_2$-$\beta_1$-operator. Whenever $\beta_2=F$, we simple say that $X$ is a $\beta_1$-space, which is equivalent in this case to the fact that the closed unit ball $\overline{B}_X\in \beta_1$. 
\end{defn}
\begin{defn}
	A Banach space $X$ is called a Gelfand-Phillips space if all limited sets in $X$ are relatively norm-compact.
\end{defn}
As a simple fact, a Banach space $X$ is Gelfand-Phillips if and only if every limited operator with range in $X$ is compact.
Further information of Gelfand-Phillips spaces can be found in \cite{LD}.
The following proposition summarizes some well known results. 
\begin{prop} \label{prop1} Let $X$ be a Banach space. Then

\begin{enumerate}

    \item $X$ is a $H$-space if and only if $X$ is finite dimensional (this is the Riesz theorem).
    \item $X$ is a $wH$-space if and only if $X$ is reflexive (this is the James theorem).
    \item  $X$ is a $L$-space if and only if $X$ is finite dimensional (this is the Josefson-Nissenzweig theorem).
    \item $X$ is a $wH$-$H$-space if and only if $X$ has the Schur property (easy to see, using the Eberlein-\u{S}mulian Theorem).
    \item $X$ is $L$-$H$-space if and only if $X$ is a Gelfand-Phillips space (simply by coincidence of definitions).
\end{enumerate}
\end{prop}

Thanks to Proposition~\ref{prop1} $(1)$ and $(3)$, a Banach space $X$ is a $L$-space if and only if it is a $H$-space. 
The following characterization of a $\beta_Y$-$\beta_X$-space is an easy consequence of Theorem~\ref{beta differentiability}.

\begin{prop} \label{prop2} Let  $X$ be a Banach space and let $\beta_1$ and $\beta_2$  be two vector bornologies on $X$ such that $\beta_1$ satisfies property $(S)$. 
Then, $X$  is a $\beta_2$-$\beta_1$-space (i.e. $\beta_2 \subset \beta_1$) if and only if for every real-valued Lipschitz function $f$ on $X$, $f$ is $\beta_2$-differentiable at some point whenever it is $\beta_1$-differentiable at this point.
\end{prop}
\begin{proof} It is enough to apply Theorem~\ref{beta differentiability} to the identity operator $Id$ on $X$.
\end{proof}
Proposition \ref{prop1} and Proposition \ref{prop2} immediately give the following characterizations of  known spaces. 
Up to the best of our knowledge, Corollary~\ref{theorem0} and Corollary~\ref{theorem00} are new results while Corollary~\ref{theorem01} and Corollary~\ref{theorem02} are known results in \cite{BFV}.
\begin{cor} \label{theorem0}  
	Let $X$ be a Banach space. Then, $X$ is a Gelfand-Phillips space if and only if G\^ateaux-differentiability (equivalently Hadamard-differentiability) and  limited-differentiability coincide for real-valued Lipschitz functions on $X$.  
\end{cor}

It is a well known result  that a Banach space $X$ is of finite dimensional if and only if G\^ateaux-differentiability and Fr\'echet differentiability coincide for real-valued Lipschitz functions. On the other hand, Corollary \ref{theorem0} shows that outside Gelfand-Phillips spaces,  G\^ateaux-differentiability  and  limited-differentiability  do not coincide in general  for real-valued Lipschitz functions on $X$. We obtain the following characterization of finite dimensional spaces in terms of  the coincide between limited-differentiability and  Fr\'echet-differentiability of real-valued Lipschitz functions.

\begin{cor} \label{theorem00}  
	Let $X$ be a Banach space. Then, $X$ is finite dimensional if and only if limited-differentiability and  Fr\'echet-differentiability coincide for real-valued Lipschitz functions on $X$.  
\end{cor}

\begin{cor} \label{theorem01}  (see \cite[Theorem 1.4]{BFV}) A Banach space $X$ is a reflexive space if and only if weak-Hadamard differentiability and Fr\'echet differentiability coincide for real-valued Lipschitz functions.
\end{cor}
\begin{cor} \label{theorem02} (see \cite[Theorem 4.1]{BFV}) A Banach space $X$ has the Schur property if and only if G\^ateaux differentiability and weak Hadamard differentiability coincide for real-valued Lipschitz functions.
\end{cor}

\vskip5mm 
\subsection{Spaceability and differentiability in the space of Lipschitz functions}
Let $\beta_X$ and $\beta_Y$ be vector bornologies on $X$ and $Y$ respectively such that $\beta_X$ satisfies Property $(S)$.
Let $T\in \mathcal{L}(X,Y)$ be a non-$\beta_Y$-$\beta_X$-operator.
Then, the set of Lipschitz functions, $\beta_X$ differentiable at $0$ such that $f\circ T$ is not $\beta_Y$-differentiable at $0$, denoted by $\mathcal{F}$, is dense in the space of Lipschitz, $\beta_X$-differentiable functions at $0$ (for the topology generated by the Lipschitz seminorm).
In what follows, we want to measure the size of $\mathcal{F}$ in an algebraic sense. 
To do this, let us introduce the following concepts.
Let $\alpha$ be a cardinal number. 
A set $A\subset X$ is said $\alpha$-lineable if $A\cup \{0\}$ contains a subspace of dimension $\alpha$. 
A set $A\subset X$ is said $\alpha$-spaceable, if $A\cup\{0\}$ contains a closed subspace of dimension $\alpha$. 
The following corollary states that the set $\mathcal{F}$ is $c$-spaceable, meaning that it contains an isometric copy of a Banach space of dimension of the continuum. 
More on lineability and spaceability can be found in \cite{ABPS,ABMP,GQ} and references therein.
The following result extends \cite[Corollary 3.9]{BFT}.
\\

\begin{cor}
	Let $X$ and $Y$ be two Banach spaces. 
	Let $\beta_X$ and $\beta_Y$ be a vector bornologies on $X$ and $Y$ respectively such that $\beta_X$ satisfies property $(S)$.
	Let $T: Y\to X$ be a $\beta_Y$-p.h. differentiable operator at $0$ such that $d_{\beta_Y} ^{ph}T(0)$ is not a $\beta_Y$-$\beta_X$-operator. 
	The set of real-valued Lipschitz functions in $\textup{Lip}_0(X)$ which are $\beta_X$-differentiable at $0$, with $\beta_X$-differential equal to $0$, but $f\circ T$ is not $\beta_Y$-differentiable at $0$, contains a subspace isometric to $\ell^\infty(\N)$, up to the function $0$.
\end{cor}

\begin{proof}
	Let us fix $R= d_{\beta_Y}^{ph} T (0)$, which is a continuous positively homogeneous function from $Y$ to $X$.
	Let $A\subset Y$ be a bounded semi-balanced set such that $A\in \beta_Y$ but the bounded, semi-balanced set $R(A)\notin \beta_X$.
	Let $\sigma>0$ and let $(x_{n})_n\subset R(A)$ be a $\sigma$-separated sequence given by Lemma~\ref{beta lipschitz}. 
	Let $(y_n)_n\subset A$ such that  $R(y_{n})=x_{n}$. 
	For each $p\in \N$ prime number define the sets ${B_{p,n}=B(\frac{x_{p^{n}}}{p^{n}},\frac{\sigma}{4p^{n}})}$ and ${B_{p}:=\cup_{n}B_{p,n}}$. 
	As in Lemma~\ref{beta lipschitz}, for each $p\in \N$, we define $f_p:X\to \R$ by 
	\[f_{p}(x)=\textup{dist}(x,X\setminus B_p),~\text{for all }x\in X,\] 
	which is $1$-Lipschitz, $\beta_X$-differentiable at $0$, with $d_{\beta_X}f(0)=0$.
	Moreover, as in the proof of Theorem~\ref{beta differentiability}, $f_p\circ T$ is not $\beta_Y$-differentiable at $0$. 
	In what follows, $(p_{i})_{i}$ stands for an enumeration of the prime numbers.
	Clearly, the supports of the functions $\{f_{p_{i}}:~i\in \N\}$ intersects only at $0$.
	Therefore, $(f_{p_{i}})_i\subset \textup{Lip}_{0}(X)$ is a sequence of linearly independent functions. 
	Moreover, if $\mu\in\ell^{\infty}(\N)$, the function
	\[ f_{\mu}(x): =\sum_{i=1}^{\infty}\mu_if_{p_i}(x), \]
	is well defined (because for each $x\in X$ there is at most one non-zero term in the series) and is $\|\mu\|_{\infty}$-Lipschitz. 
	On the other hand, since $f_\mu=\mu_i f_{p_i}$ on $\textup{supp}(f_{p_i})$, we have that $\textup{Lip}(f_\mu)\geq \sup\{|\mu_i|:~i\in\N \}$.
	That is, the linear operator $L:\ell^{\infty}(\N)\to\textup{Lip}_{0}(X)$ given by $L\mu = f_{\mu}$ is an isometry. 
	Since $(x_n)_n$ is a $\sigma$-separated sequence, Proposition~\ref{sigma/4} implies that $0\in\textup{core} (X\setminus \cup_k B_{p_i^k})$ and, thanks to Proposition~\ref{coregateaux}, $f_\mu$ is G\^ateaux-differentiable at $0$. 
	Moreover, an analogous argument of the proof of Lemma~\ref{beta lipschitz} shows that $f_\mu$ is, in fact, $\beta_X$-differentiable at $0$.
	However, if $\mu\in \ell^\infty(\N)$ and $\mu\neq 0$, $f_{\mu}\circ T$ is not $\beta_Y$-differentiable at $0$. 
	Indeed, observe first that the only candidate for $\beta_Y$-(p.h.) differential of $f_\mu\circ T$ at $0$ is the functional $0$. 
	Let $k\in \N$ such that $\mu_k\neq 0$. 
	Since $A\in \beta_Y$ and recalling that $R=d_{\beta_Y}^{ph}T(0)$, for any $n\in \N$ we have that
	
	\begin{align*}
	 \dfrac{f_\mu \circ T(y_{p_{k}^{n}}/p_{k}^{n})-f_\mu\circ T(0)}{1/p_{k}^{n}}   
		&= \dfrac{f_\mu \circ T(y_{p_{k}^{n}}/p_{k}^{n})-f_\mu \circ R(y_{p_{k}^{n}}/p_{k}^{n})}{1/p_{k}^{n}}   +\dfrac{f_\mu \circ R(y_{p_{k}^{n}}/p_{k}^{n})-f_\mu\circ T(0)}{1/p_{k}^{n}} \\
		&\geq - \textup{Lip}(f_\mu)\left\| \dfrac{ T(y_{p_{k}^{n}}/p_{k}^{n})-R(y_{p_{k}^{n}}/p_{k}^{n})}{1/p_{k}^{n}}  \right \| + p_k^n f_\mu (x_{p_k^n}/p_k^n )\\
		&= \dfrac{\sigma}{4} -\textup{Lip}(f_\mu)\left\| \dfrac{ T(y_{p_{k}^{n}}/p_{k}^{n})-R(y_{p_{k}^{n}}/p_{k}^{n})}{1/p_{k}^{n}}  \right \|.
	\end{align*}
	
	Therefore, since $R$ is the $\beta_Y$-p.h. differential of $T$ and $(y_{p_k^n})\subset A\in \beta_Y$, by sending $n$ to infinity we obtain that the last expression tends to $\sigma/4$. 
	Hence, $f_\mu\circ T$ is not $\beta_Y$ differentiable at $0$.
\end{proof}

\subsection{Non-linear strict cosingularity and almost weakly compactness.}

In this last subsection we give an extension of the so-called Bourgain-Diestel Theorem, see \cite{BD} or \cite{BFT}. 
Recall from \cite{BD} (resp. from \cite{H}) that a bounded linear operator $T : Y\to X$  between Banach spaces $Y$ and $X$  is called
strict cosingular (resp. almost weakly compact) if the only Banach spaces $E$ for which one can find a  bounded linear operator $q: X \to E$ for which $q \circ T$ is surjective,  are finite dimensional (resp. are reflexive spaces). The strict cosingularity and almost weakly compactness, are generalizations of compact and weakly compact operators respectively. 
A generalization of (linear) strict cosingularity can be found in \cite{wH}. \\

We give in the following definitions a non-linear extension of the notion of strict cosingularity and almost weakly compactness respectively. 
Recall the notion of co-Lipschitz mapping introduced in  \cite{BJLPS,JLPS}: a mapping $L : X\to E$ is called co-Lipschitz if there is $c>0$ such that 
\[L(x)+crB_E \subset L(x + rB_X ),~ \text{for all }x\in X,~r>0, \]
Notice that a bounded linear operator $L:X\to E$ is co-Lipschitz if and only if it is surjective by the Open-mapping theorem. 

\begin{defn} 
	Let  $X$ and $Y$ be two Banach spaces. An operator (not necessarily linear) $T : Y\to X$ is said to be strict cosingular (resp. strong strict cosingular) if the following property holds:  for every Banach space $E$, if there exists a linear bounded operator $q: X \to E$ (resp. a Lipschitz map, limited differentiable at $0$ with $q(0)=0$), such $q \circ T$ is  co-Lipschitz, then $E$ is a finite dimentional.
\end{defn}
\begin{defn} 
	Let  $X$ and $Y$ be two Banach spaces. An operator (not necessarily linear) $T : Y\to X$ is said to be almost weakly compact (resp. strong almost weakly compact) if the following property holds:  for every Banach space $E$, if there exists a linear bounded operator $q: X \to E$ (resp. a Lipschitz map, weak Hadamard differentiable at $0$ with $q(0)=0$), such $q \circ T$ is  co-Lipschitz, then $E$ is a reflexive space.
\end{defn}

Clearly, the strong strict cosingularity (resp. the strong almost weakly compactness) implies the strict cosingularity (the almost weakly compactness). Bourgain and Diestel proved in \cite{BD} that a limited linear bounded operator (i.e. linear $L$-operator) is strict cosingular. 
Theorem \ref{beta-cosingular} gives an extension to the  Bourgain-Diestel result, where the operator $T$ is not assumed to be necessarily linear but only its (generalized positively homogeneous) differential at $0$ is supposed to be a limited operator. Moreover, Theorem \ref{beta-cosingular} gives the strong strict cosingularity. A similar result concerning  weak Hadamard bornologies and strong almost weakly compactness is given  in Theorem \ref{beta-cosingular-bis}.  

\begin{thm} \label{beta-cosingular} Let  $X$ and $Y$ be two Banach spaces. Let $T: Y\to X$ be a (not necessarily linear) operator $F$-p.h. differentiable at some point $y\in Y$ such that $d_F^{ph} T(y)$ is a limited operator. Then, $T$ is strong strict cosingular.
\end{thm}
\begin{thm} \label{beta-cosingular-bis} Let  $X$ and $Y$ be two Banach spaces. Let $T: Y\to X$ be a (not necessarily linear) operator $F$-p.h. differentiable at some point $y\in Y$ such that $d_{F}^{ph} T(y)$ is a weakly compact operator. Then, $T$ is strong almost weakly compact.
\end{thm}
The proofs of Theorem~\ref{beta-cosingular} and Theorem~\ref{beta-cosingular-bis} are given as an immediate consequence of  Lemma~\ref{theorem1} at the end of the paper. 
Using the above theorems, together with Theorem~\ref{theorem-beta-beta-op}, we get the following corollaries.

\begin{cor} Let  $X$ and $Y$ be two Banach spaces. Let $T: Y\to X$ be a (not necessarily linear) operator $F$-p.h. differentiable at some point $y\in Y$. Suppose that $T$ is a $L$-operator. Then, $T$ is strong strict cosingular.
\end{cor}
\begin{cor} Let  $X$ and $Y$ be two Banach spaces. Let $T: Y\to X$ be a (not necessarily linear) operator $F$-p.h. differentiable at some point $y\in Y$. Suppose that $T$ is a $wH$-operator. Then, $T$ is strong almost weakly compact.
\end{cor}
In order to state  Lemma~\ref{theorem1} in all its generality, let us continue with the following definition which generalize the notion of co-Lipschitz map.
\begin{defn} 
	Let  $X$ and $E$ be two Banach spaces and $\beta_X$, $\beta_E$ be two bornologies on $X$ and $E$ respectively. Let $q : X \to E$ be a mapping. 
	We say that $q$ is $\beta_E$-$\beta_X$-co-Lipschitz, if for every $A\in \beta_E$ and every $\varepsilon >0$, there exists $B\in \beta_X$ such that 
	\[ \limsup_{t\to 0^+}\sup_{h\in A} \textup{dist}(h,\frac{q(tB)}{t}):=\limsup_{t\to 0^+}\sup_{h\in A} \inf_{x\in B}\|\frac{q(tx)}{t}-h\|\leq \varepsilon.\]
Whenever $\beta_E$ is the Fr\'echet bornology, we simple say that $q$ is $\beta_X$-co-Lipschitz.
\end{defn}
\begin{rem} \label{Drem} Let  $X$ and $E$ be two Banach spaces and $\beta_X$, $\beta_E$ be two bornologies on $X$ and $E$ respectively. Let $q : X \to E$ be a mapping. Suppose that for every $A\in \beta_E$ and every $\varepsilon >0$, there exists $B\in \beta_X$ such that 
\[rA \subset q(rB ) + \varepsilon r B_Z,~ \text{for all }r>0, \]
then, $q$ is $\beta_E$-$\beta_X$-co-Lipschitz. Notice also that, if $q:X\to E$ is a co-Lipschitz map and $q(0)=0$, then $q$ is $F$-$F$-co-Lipschitz.
\end{rem}

\begin{lem} \label{theorem1} 
	Let $X$, $Y$ and $E$ be Banach spaces, let $\beta_X$ and $\beta_Y$ be vector bornologies on $X$ and $Y$ respectively and $\beta_E$ and  $\beta'_E$ be two vector bornologies on $E$. 
	Assume that $\beta_X$ and $\beta_E$ satisfy property $(S)$.
	Let $T: Y\to X$ be a $\beta_Y$-p.h. differentiable at $0$ mapping such that $T(0)=0$ and $d_{\beta_Y} ^{ph}T(0)$ is a $\beta_Y$-$\beta_X$-operator. 
	Suppose that there exists a Lipschitz  mapping $q : X  \to E$, such that $q (0)=0$, $q$ is $\beta_X$-p.h. differentiable at $0$, $d_{\beta_X} ^{ph}q(0) $ is a $\beta_X$-$\beta_E$-operator and  $q \circ T$ is $\beta'_E$-$\beta_Y$-co-Lipschitz. 
	Then, $E$ is a $\beta'_E$-$\beta_E$-space. 
	
\end{lem}
\begin{proof}  
	Suppose that $E$ is not a $\beta'_E$-$\beta_E$-space, then by Proposition~\ref{prop2}, there exists a positive Lipschitz function $f : E\to \R$ such that $f$ is $\beta_E$-differentiable at $0$ but not  $\beta'_E$-differentiable at $0$, with $f(0)=0$ and $d_{\beta_E}f(0)=0$. \\
	
	Since $f$ is not $\beta'_E$-differentiable at $0$, then there exists $A\in \beta'_E$ such that
\begin{align}\label{equation 1}
\exists \varepsilon >0, \forall\delta >0, \exists t\in(0,\delta)~ \textnormal{ such that } \sup_{h\in A} \frac{ f(th)}{t}\geq \varepsilon.
\end{align}
Since $q \circ T$ is $\beta'_E$-$\beta_Y$-co-Lipschitz, there exists  $B\in \beta_Y$ such that  
\begin{align}\label{equation 2}
\limsup_{t\to 0^+}\sup_{h\in A}\inf_{x\in B}\|\frac{q\circ T (tx)}{t}-h\|\leq \frac{\varepsilon}{2\textup{Lip}(f)}.
\end{align}
On the other hand, since $f$ is Lipschitz, we have for all $y\in B$, for all $h\in A$ and for all $t>0$
\begin{align}\label{equation 3}
		t^{-1} f\circ q \circ T(ty)  &\geq   t^{-1} f(th) -\textup{Lip}(f)\|\frac{q \circ T(ty)}{t}-h\|.
	\end{align}
	Since $f$ is a positive function and $f(q (T(0)))=0$, the only candidate for $\beta_Y$-derivative of $f\circ q \circ T$ at $0$ is $0$. 
	Thanks to \eqref{equation 1} and \eqref{equation 3}, for any $\delta>0$, there is $t\in(0,\delta)$ such that
	\begin{align*}
		\sup_{y\in B} t^{-1} f\circ q \circ T(ty)  &\geq  \sup_{h\in A} \left( t^{-1} f(th) -\textup{Lip}(f)\inf_{y\in B}\|\frac{q \circ T(ty)}{t}-h\| \right) \notag \\
&\geq  \sup_{h\in A} t^{-1} f(th) -\textup{Lip}(f)\sup_{h\in A}\inf_{y\in B}\|\frac{q \circ T(ty)}{t}-h\| \notag \\
		&\geq \varepsilon-\textup{Lip}(f)\sup_{h\in A}\inf_{y\in B}\|\frac{q \circ T(ty)}{t}-h\|.
	\end{align*}
Therefore, combining the last inequality and \eqref{equation 2}, it follows that 
\begin{align*}
		\limsup_{t\to 0^+} \sup_{y\in B} t^{-1} f\circ q \circ T(ty)  &\geq   \varepsilon-\textup{Lip}(f)\liminf_{t\to 0^+} \sup_{h\in A}\inf_{y\in B}\|\frac{q \circ T(ty)}{t}-h\|\notag \\
           &\geq \varepsilon/2.
	\end{align*}
	This shows that $f\circ q \circ T$ is not $\beta_Y$-differentiable at $0$. \\

On the other hand,  using Theorem~\ref{beta differentiability}, $f\circ q $ is  $\beta_X$-differentiable at $0=T(0)$ since $d_{\beta_X} ^{ph}q(0)$  is a $\beta_X$-$\beta_E$-operator. Now, since $f\circ q : X \to \R$ is Lipschitz, $\beta_X$-differentiable at $0=T(0)$ and $f\circ q \circ T$ is not $\beta_Y$-differentiable at $0$, we get from Theorem~\ref{beta differentiability} that $d_{\beta_Y} ^{ph}T(0)$  is not a $\beta_Y$-$\beta_X$-operator, which is a contradiction.  
\end{proof}

Now, we give the proofs of Theorem~\ref{beta-cosingular} and Theorem~\ref{beta-cosingular-bis}.
\begin{proof}[Proof of Theorem~\ref{beta-cosingular}] We apply Lemma~\ref{theorem1} to the function $S=T(\cdot+y)-T(y)$ with $\beta_Y=\beta_E'=F$ and $\beta_X=\beta_E=L$ (limited bornology), using Remark \ref{Drem} and the observation that $d_{L} q(0) $, as a linear bounded operator, is always $L$-$L$-operator. We get that $E$ is an $L$-space, that is a finite dimensional space by Proposition~\ref{prop1}.
\end{proof}

\begin{proof}[Proof of Theorem~\ref{beta-cosingular-bis}] We apply Lemma~\ref{theorem1} to the function $S=T(\cdot+y)-T(y)$ with $\beta_Y=\beta_E'=F$ and $\beta_X=\beta_E=wH$ (weak Hadamard bornology), using Remark \ref{Drem} and the observation   that $d_{wH} q(0) $, as a linear bounded operator, is always $wH$-$wH$-operator. We get that $E$ is a $wH$-space, that is a reflexive space by Proposition~\ref{prop1}.
\end{proof}
\section*{Acknowledgement}
This research has been conducted within the FP2M federation (CNRS FR 2036) and  SAMM Laboratory of the University Paris Panthéon-Sorbonne. The second author was partially supported by ANID-PFCHA/Doctorado Nacional/2018-21181905 and by CMM (UMI CNRS 2807), Basal grant: AFB170001.

\bibliographystyle{amsplain}

\end{document}